\documentclass[11pt]{amsart}
\usepackage{amssymb,amsmath,amsthm,latexsym,stmaryrd}

\usepackage[T1]{fontenc}
\usepackage{lmodern}

\usepackage[pagebackref,colorlinks=true]{hyperref}  
\newtheorem{theorem}{Theorem}[section]

\newtheorem{corollary}[theorem]{Corollary}

\newtheorem{definition}[theorem]{Definition}

\newtheorem{proposition}[theorem]{Proposition}
\newtheorem{remark}[theorem]{Remark}

\newcommand{\ie}{i.\hspace{.5pt}e.\ }
\newcommand{\f}{\varphi}
\newcommand{\g}{\tilde{g}}
\newcommand{\n}{\nabla}
\newcommand{\nn}{\tilde{\n}}
\newcommand{\M}{(M,\A\f,\A\xi,\A\eta,\A{}g)}

\newcommand{\I}{\iota}

\newcommand{\R}{\mathbb R}

\newcommand{\F}{\mathcal{F}}
\newcommand{\LL}{\mathcal{L}}
\newcommand{\ta}{\theta}
\newcommand{\om}{\omega}
\newcommand{\lm}{\lambda}

\newcommand{\al}{\alpha}
\newcommand{\bt}{\beta}

\DeclareMathOperator{\D}{d} 

\DeclareMathOperator{\Div}{div} 
\DeclareMathOperator{\Span}{span} 

\newcommand{\thmref}[1]{Theorem~\ref{#1}}

\newcommand{\corref}[1]{Corollary~\ref{#1}}
\newcommand{\propref}[1]{Proposition~\ref{#1}}
\newcommand{\remref}[1]{Remark~\ref{#1}}

\newcommand{\A}{\allowbreak{}}

\begin{document}

\title{Ricci-like Solitons on Almost Contact B-Metric Manifolds}

\author[M. Manev]{Mancho Manev
}

\address[M. Manev]{University of Plovdiv Paisii Hilendarski,
Faculty of Mathematics and Informatics, Department of Algebra and
Geometry, 24 Tzar Asen St., Plovdiv 4000, Bulgaria
\&
Medical University of Plovdiv, Faculty of Public Health,
Department of Medical Informatics, Biostatistics and E-Learning, 15A Vasil Aprilov Blvd.,
Plovdiv 4002, Bulgaria} 
\email{mmanev@uni-plovdiv.bg}

\begin{abstract}
Ricci-like solitons with potential Reeb vector field are introduced and studied
on almost contact B-metric manifolds. 
The cases of Sasaki-like manifolds and torse-forming potentials have been considered.
In these cases, it is proved that the manifold admits 
a Ricci-like soliton 
if and only if 
the structure is Einstein-like. 
Explicit examples of Lie groups as 3- and 5-dimensional manifolds with the structures studied are 
provided.
\end{abstract}



\date{December 17, 2019}

\dedicatory{Dedicated to the memory of the author's irreplaceable colleague and friend \\ Prof. Dimitar Mekerov.}


\maketitle



\section{Introduction}

In differential geometry, the Ricci flow is an intrinsic geometric flow. 
It is a process of deformation of the Riemannian metric, independent of any embedding or immersion. 
The Ricci flow is given by an evolution equation for the metric and the associated Ricci tensor on a Riemannian manifold depending on a variable called time. 
Ricci flows, formally analogous to the diffusion of heat, are among the important tools of theoretical physics. 

In 1982, R.\,S. Hamilton \cite{Ham82} introduced the concept of Ricci flow and proved its
existence. This idea is contributed to the proof of Thurston's Geometrization Conjecture and consequently the Poincar\'e Conjecture.

Ricci soliton is a special solution of the Ricci flow equation and it
moves only by a one-parameter family of diffeomorphism and scaling. 
Ricci solitons represent a natural generalization of Einstein metrics on a Riemannian
manifold, being generalized fixed points of Hamilton's Ricci flow, considered as a dynamical system.

%
%

In 2008, R. Sharma \cite{Shar} initiated the study of Ricci solitons in contact Riemannian geometry, in particular on K-contact manifolds. After that
Ricci solitons have been studied on different kinds of almost contact metric manifolds, e.g.
$\al$-Sasakian \cite{IngBag}, trans-Sasakian \cite{BagIng13}, Kenmotsu
\cite{BagIngAsh13}, \cite{NagPre}, nearly Kenmotsu \cite{AyaYil}, $f$-Kenmotsu \cite{GalCra},  etc. 
In paracontact geometry, Ricci solitons are first investigated 
by G. Calvaruso and D. Perrone \cite{CalPer}.

As a generalization of a Ricci soliton, the notion of $\eta$-Ricci soliton has been introduced by J.\,T. Cho and M. Kimura \cite{ChoKim}.
In the context of
paracontact geometry $\eta$-Ricci solitons have been investigated in 
\cite{Bla15}, \cite{Bla16}, \cite{BlaCra}, \cite{PraHad}.

Although this topic was first studied in Riemannian geometry, 
in recent years Ricci solitons and their generalizations 
have also been studied in pseudo-Riemannian manifolds, 
mostly with Lorentzian metrics 
(\cite{BagIng12}, \cite{Bla16}, \cite{BlaPer}, \cite{BlaPerAceErd}, \cite{Mat}).

Usually, an almost contact manifold is equipped with
a metric called compatible, which can be Riemannian or pseudo-Riemannian. 
Then, the contact endomorphism 
acts as an isometry with respect to the 
metric on the contact distribution.
The alternative possibility is when the contact endomorphism 
acts as an anti-isometry. 
In this case, the metric is necessarily pseudo-Riemannian and it is known as a B-metric.
Other important characteristic of almost contact B-metric structure which differs
it from the metric one is that the associated $(0,2)$-tensor of the B-metric is also a
B-metric. 
Thus, it is obtained an almost contact B-metric manifold.
In another point of view, it is an almost contact complex Riemannian manifold.
The differential geometry of almost contact B-metric manifolds has been studied since 1993 \cite{GaMiGr}, \cite{ManGri93}.

In the present paper, our goal is to introduce and investigate a generalization of the Ricci soliton 
compatible with the almost contact B-metric structure, and its potential to be the Reeb vector field. 
We are studying these objects on some important kinds of manifolds under consideration: Einstein-like, 
Sasaki-like and having a torse-forming Reeb vector field.
In relation with the proved assertions, we comment explicit examples in dimension 3 and 5 of Lie groups 
with the structure studied.

\section{Almost Contact B-Metric Manifolds}

Let us consider an \emph{almost contact B-metric manifold} denoted by $\M$. This means that $M$
is a $(2n+1)$-dimensional differentiable manifold with an almost
contact structure $(\f,\xi,\eta)$, where $\f$ is an endomorphism
of the tangent bundle $TM$, $\xi$ is a Reeb vector field and $\eta$ is its dual contact 1-form. Moreover,
$M$ is equipped with a pseu\-do-Rie\-mannian
metric $g$  of signature $(n+1,n)$, such that the following
algebraic relations are satisfied: \cite{GaMiGr}
\begin{equation}\label{strM}
\begin{array}{c}
\f\xi = 0,\qquad \f^2 = -\I + \eta \otimes \xi,\qquad
\eta\circ\f=0,\qquad \eta(\xi)=1,\\
g(\f x, \f y) = - g(x,y) + \eta(x)\eta(y),
\end{array}
\end{equation}
where $\I
$ is the identity transformation on $\Gamma(TM)$. In the latter equality and further, $x$, $y$, $z$, $w$ will stand for arbitrary elements of $\Gamma(TM)$ or vectors in the tangent space $T_pM$ of $M$ at an arbitrary
point $p$ in $M$.

Some immediate consequences of \eqref{strM} are the following equations
\begin{equation}\label{strM2}
\begin{array}{ll}
g(\f x, y) = g(x,\f y),\qquad &g(x, \xi) = \eta(x),
\\
g(\xi, \xi) = 1,\qquad &\eta(\n_x \xi) = 0,
\end{array}
\end{equation}
where $\n$ is the Levi-Civita connection of $g$.

The associated metric $\g$ of $g$ on $M$ is defined by
\(\g(x,y)=g(x,\f y)+\eta(x)\eta(y)\).  The manifold
$(M,\f,\xi,\eta,\g)$ is also an almost contact B-metric manifold.
The B-metric $\g$ is also of signature $(n+1,n)$.
The Levi-Civita connection of $\g$ is denoted by
$\nn$.

A classification of almost contact B-metric manifolds, which contains eleven basic classes $\F_1$, $\F_2$, $\dots$, $\F_{11}$, is given in
\cite{GaMiGr}. This classification is made with respect
to the tensor $F$ of type (0,3) defined by
\begin{equation}\label{F=nfi}
F(x,y,z)=g\bigl( \left( \nabla_x \f \right)y,z\bigr).
\end{equation}
The following identities are valid:
\begin{equation}\label{F-prop}
\begin{array}{l}
F(x,y,z)=F(x,z,y)\\
\phantom{F(x,y,z)}
=F(x,\f y,\f z)+\eta(y)F(x,\xi,z)
+\eta(z)F(x,y,\xi),\\
F(x,\f y, \xi)=(\n_x\eta)y=g(\n_x\xi,y).
\end{array}
\end{equation}

The special class $\F_0$,
determined by the condition $F=0$, is the intersection of the basic classes and it is known as the
class of the \emph{cosymplectic B-metric manifolds}.

Let $\left\{e_i;\xi\right\}$ $(i=1,2,\dots,2n)$ be a basis of
$T_pM$ and let $\left(g^{ij}\right)$ be the inverse matrix of the
matrix $\left(g_{ij}\right)$ of $g$. Then the following 1-forms
are associated with $F$:
\begin{equation}\label{t}
\theta(z)=g^{ij}F(e_i,e_j,z),\quad
\theta^*(z)=g^{ij}F(e_i,\f e_j,z), \quad \omega(z)=F(\xi,\xi,z).
\end{equation}
These 1-forms are known also as the Lee forms of the considered manifold. Obviously, the identities
$\om(\xi)=0$ and $\ta^*\circ\f=-\ta\circ\f^2$ are always valid.

\subsection{Sasaki-like almost contact B-metric manifolds}

In \cite{IvMaMa45}, it is defined the class of \emph{Sasaki-like spaces} in the set of almost
contact B-metric manifolds (also known as almost contact complex Riemannian manifolds) by the condition its complex cone to be a K\"ahler-Norden manifold.
Then, these Sasaki-like spaces are determined by the condition
\begin{equation}\label{defSl}
\begin{array}{l}
\left(\nabla_x\f\right)y=-g(x,y)\xi-\eta(y)x+2\eta(x)\eta(y)\xi,
\end{array}
\end{equation}
which may be rewritten as $\left(\nabla_x\f\right)y=g(\f x,\f y)\xi+\eta(y)\f^2 x$.

\begin{remark}\label{rem:Sl}
Obviously, the considered manifolds of Sasaki-like type form a subclass of the basic class $\F_4$ 
of the classification in \cite{GaMiGr} and they
have Lee forms of the form $\ta=2n\,\eta$, $\ta^*=\om=0$. 
\end{remark}

Moreover,  
the following identities are valid for them \cite{IvMaMa45}
\begin{equation}\label{curSl} 
\begin{array}{ll}
\n_x \xi=-\f x, \qquad &\left(\n_x \eta \right)(y)=-g(x,\f y),\\
R(x,y)\xi=\eta(y)x-\eta(x)y, \qquad &R(\xi,y)\xi=\f^2y, \\[0pt]
\rho(x,\xi)=2n\, \eta(x),\qquad 				&\rho(\xi,\xi)=2n,
\end{array}
\end{equation}
where $R$ and $\rho$ stand for the curvature tensor and the Ricci tensor.
As a consequence we have
\begin{equation*}
R(\xi,y)z=g(y,z)\xi-\eta(z)y, \qquad \eta\bigl(R(x,y)z\bigr)=\eta(x)g(y,z)-\eta(y)g(x,z);
\end{equation*}
\begin{equation*}
\begin{array}{lll}
\nabla_{\xi}\xi=0, \qquad &\nabla_{\xi}\f=0, \qquad &\n_{\xi} \eta=0,\\
\D\eta=0,\qquad & R(\f x,\f y)\xi=0, \qquad & \rho(\f x,\xi)=0.
\end{array}
\end{equation*}

Obviously, the sectional curvature of an arbitrary $\xi$-section $\al=\Span\{x,\xi\}$ is $k(x,\xi)=1$.

\subsubsection{Example of a Sasaki-like manifold}\label{ex1}
As Example 2 in \cite{IvMaMa45}, it is considered 
the Lie group $G$ of
dimension $5$ 
 with a basis of left-invariant vector fields $\{e_0,\dots, e_{4}\}$
and the corresponding Lie algebra is 
 defined by the commutators
\begin{equation}\label{comEx1}
\begin{array}{ll}
[e_0,e_1] = p e_2 + e_3 + q e_4,\quad &[e_0,e_2] = - p e_1 -
q e_3 + e_4,\\[0pt]
[e_0,e_3] = - e_1  - q e_2 + p e_4,\quad &[e_0,e_4] = q e_1
- e_2 - p e_3,\qquad p,q\in\R.
\end{array}
\end{equation}
Then, $G$ is equipped with an almost contact B-metric structure by
\begin{equation}\label{strEx1}
\begin{array}{l}
g(e_0,e_0)=g(e_1,e_1)=g(e_2,e_2)=-g(e_{3},e_{3})=-g(e_{4},e_{4})=1,
\\[0pt]
g(e_i,e_j)=0,\quad
i,j\in\{0,1,\dots,4\},\; i\neq j,
\\[0pt]
\xi=e_0, \quad \f  e_1=e_{3},\quad  \f e_2=e_{4},\quad \f  e_3=-e_{1},\quad \f  e_4=-e_{2}.
\end{array}
\end{equation}

The components of the
Levi-Civita connection are the following
\begin{equation}\label{nijEx1}
\begin{array}{c}
\begin{array}{ll}
\n_{e_0} e_1 = p e_2 + q e_4,\qquad &
\n_{e_0} e_2 = - p e_1 - q e_3, \\
\n_{e_0} e_3 = - q e_2 + p e_4,\qquad &
\n_{e_0} e_4 = q e_1 - p e_3,
\end{array}
\\
\begin{array}{c}
\n_{e_1}e_0 = - e_3,\quad
\n_{e_2} e_0 = - e_4,\quad
\n_{e_3} e_0 = e_1,\quad
\n_{e_4}e_0 =e_2,
\end{array}
\\
\begin{array}{c}
\n_{e_1}e_3 = \n_{e_2} e_4 = \n_{e_3} e_1 = \n_{e_4}e_2 = - e_0.
\end{array}
\end{array}
\end{equation}

It is verified that the constructed almost contact B-metric manifold
$(G,\f,\allowbreak{}\xi,\allowbreak{}\eta,\allowbreak{}g)$ is Sasaki-like.

\subsection{Einstein-like almost contact B-metric manifolds}

Let us introduce the following
\begin{definition}
An almost contact B-metric manifold $\M$ is said to be
\emph{Einstein-like} if its Ricci tensor $\rho$ satisfies
\begin{equation}\label{defEl}
\begin{array}{l}
\rho(x,y)=a\,g(x,y) +b\,\g(x,y) +c\,\eta(x)\eta(y)
\end{array}
\end{equation}
for some triplet of constants $(a,b,c)$.
\end{definition}

In particular, when $b=0$ and $b=c=0$, the manifold is called an \emph{$\eta$-Einstein manifold} and an \emph{Einstein manifold}, respectively. Other particular cases of Einstein-like types of the considered manifolds are given in \cite{ManNak15} and \cite{HMMek}.

As a consequence of \eqref{defEl} we obtain that the corresponding scalar curvature is the following constant 
\begin{equation}\label{tauEl}
\begin{array}{l}
\tau=(2n+1)a +b +c.
\end{array}
\end{equation}

\begin{remark}
The Reeb vector field $\xi$ of an Einstein-like $\M$ is an eigenvector 
of the Ricci operator $Q=\rho^{\sharp}$, \ie $g(Qx,y)=\rho(x,y)$,
with the corresponding eigenvalue $a+b+c$.  
\end{remark}

Then, applying \eqref{strM} and \eqref{strM2} into \eqref{defEl}, we obtain the truthfulness of the following 
\begin{proposition}
The Ricci tensor $\rho$ of an Einstein-like manifold $\M$ has the following properties:
\begin{equation}\label{roEl1}
\begin{array}{l}
\begin{array}{l}
\rho(\f x,\f y)=-\rho(x,y)+(a+b+c)\eta(x)\eta(y),
\end{array}\\
\begin{array}{ll}
\rho(\f x,y)=\rho(x,\f y), \qquad  & \rho(\f x,\xi)=0, 
\\
\rho(x,\xi)=(a+b+c)\eta(x), \qquad  & \rho(\xi,\xi)=a+b+c,
\end{array}
\end{array}
\end{equation}
\begin{equation}\label{roEl2}
\begin{array}{l}
\left(\n_x\rho\right)(y,z)=b\,g\bigl(\left(\n_x\f\right)y,z\bigr) \\
\phantom{\left(\n_x\rho\right)(y,z)=}
+(b+c)\left\{g(\n_x \xi,y)\eta(z) +g(\n_x \xi,z)\eta(y)\right\}.
\end{array}
\end{equation}
\end{proposition}


\begin{proposition}
Let $\M$, $\dim M=2n+1$, have a scalar curvature $\tau$ and let the manifold be Sasaki-like and Einstein-like with a triplet of constants $(a,b,c)$. 
Then we have
\begin{equation}\label{tauElSl2}
a+b+c=2n,\qquad \tau=2n(a+1),
\end{equation}
\begin{equation}\label{roElSl}
\begin{array}{l}
\left(\n_x\rho\right)(y,z)=-(b+c)\left\{g(x,\f y)\eta(z) +g(x,\f z)\eta(y)\right\}\\
\phantom{\left(\n_x\rho\right)(y,z)=}{}
+b\left\{g(\f x,\f y)\eta(z) +g(\f x,\f z)\eta(y)\right\}.
\end{array}
\end{equation}
\begin{equation}\label{roElSl3}
\begin{array}{c}
a=2n+\frac{1}{2n}(\Div^*{\rho})(\xi),\qquad
b=-\frac{1}{2n}(\Div{\rho})(\xi),\\
c=\frac{1}{2n}\bigl\{(\Div{\rho})(\xi)-(\Div^*{\rho})(\xi)\bigr\},
\end{array}
\end{equation}
where $\Div$ and $\Div^*$ denote the divergence with respect to $g$ and $\g$, respectively.
\end{proposition}
\begin{proof}
Bearing in mind the last equalities of \eqref{curSl} and \eqref{roEl1}, as well as \eqref{tauEl}, we obtain \eqref{tauElSl2}.
The expression \eqref{roElSl} follows from \eqref{defSl} and \eqref{roEl2}.
Taking the traces over $x$ and $y$ in \eqref{roElSl} with respect to both metrics and using $\g^{ij}=-\f^i_kg^{ik}+\xi^i\xi^j$,
we get \eqref{roElSl3}.
\end{proof}

\begin{corollary}\label{cor:ElSl}
 Let $\M$ be Sasaki-like and Einstein-like. Then we have:
\begin{enumerate}
\item[(i)] It is scalar-flat if and only if $a=-1$.  

\item[(ii)] It is Ricci-symmetric if and only if it is an Einstein manifold.  

\item[(iii)] Its Ricci tensor  is $\eta$-parallel  and parallel along $\xi$. 

\item[(iv)] It is $\eta$-Einstein if and only if $\Div Q\in\ker\eta$.

\item[(v)] It is Einstein if and only if $\Div{Q}, \Div^*{Q}\in\ker\eta$.

\end{enumerate}
\end{corollary}
\begin{proof}
The equivalence in (i) is obvious from \eqref{tauElSl2},
and (ii) is easily deduced from \eqref{roElSl}. 
The assertion (iii) is valid since \eqref{roElSl} yields $\left(\n\rho\right)|_{\ker\eta}=0$ and  $\n_\xi\rho=0$.
The statement in (iv) is true because of the value of $b$ in \eqref{roElSl3} and the identities 
$\left(\n_x \rho\right)(y,z)=g\left(\left(\n_x Q\right)y,z\right)$
and
$(\Div\rho)(z)=g(\Div Q,z)$. Similarly, we establish the truthfulness of (v).
\end{proof}

Bearing in mind \eqref{tauElSl2}, we obtain immediately that the scalar curvature 
of any Einstein Sasaki-like manifold $\M$ is $\tau=2n(2n+1)$.

\subsubsection{Example of an Einstein-like manifold}\label{ex1-E}
We consider the example of almost contact B-metric manifold $(G,\f,\allowbreak{}\xi,\allowbreak{}\eta,\allowbreak{}g)$ from \S\,\ref{ex1}.

Taking into account \eqref{comEx1}, \eqref{strEx1} and \eqref{nijEx1}, 
we compute the components of the curvature tensor $R_{ijkl}=R(e_i,e_j,e_k,e_l)$ 
and those of the Ricci tensor  $\rho_{ij}=\rho(e_i,e_j)$.
The non-zero of them are determined by the following equalities  and the property $R_{ijkl}=-R_{jikl}=-R_{ijlk}$: 
\begin{equation}\label{Rex1}
\begin{array}{l}
R_{0110}=R_{0220}=-R_{0330}=-R_{0440}=1,\\
R_{1234}=R_{1432}=R_{2341}=R_{3412}=1,\\
R_{1331}=R_{2442}=1,\qquad
\rho_{00}=4.
\end{array}
\end{equation}

Bearing in mind \eqref{defEl}, we establish that $(G,\f,\allowbreak{}\xi,\allowbreak{}\eta,\allowbreak{}g)$
is $\eta$-Einstein, because we obtain $\rho=4\eta\otimes\eta$ and the constants are 
\begin{equation}\label{abcS}
(a,b,c)=(0,0,4).
\end{equation}

\subsection{Almost contact B-metric manifolds with a torse-forming Reeb vector field}\label{sec:tf}

Let us recall, a vector field $\xi$ is called \emph{torse-forming} if $\n_x \xi=f\,x+\al(x)\xi$, where $f$ is a smooth function and $\al$ is an 1-form on the manifold. Taking into account the last equality in \eqref{strM2}, the 1-form $\al$ on an almost contact B-metric manifold $\M$ is  determined by $\al=-f\,\eta$ and then we have the following mutually equivalent equalities
\begin{equation}\label{tf}
\begin{array}{l}
		\n_x \xi=-f\,\f^2x,\qquad \left(\n_x \eta \right)(y)=-f\,g(\f x,\f y).
		\end{array}
\end{equation}

If $f = 0$, the vector field $\xi$ in \eqref{tf} is
a parallel vector field. 
This case is trivial and we omit it in our considerations, \ie we assume that $f$ is not identically zero.


Conditions \eqref{tf} imply $\ta^*(\xi)=2n\,f$ and $\ta(\xi)=\om=0$, using \eqref{F-prop} and \eqref{t}.
Then, taking into account the components of $F$ in the basic classes $\F_i$, $i\in\{1,\dots,11\}$, given in \cite{HM1}, 
we deduce
\begin{remark}\label{rem:tf}
Manifolds $\M$ with torse-forming $\xi$ belong to 
the class $\F_1\oplus\F_2\oplus\F_3\oplus\F_5\oplus\F_6\oplus\F_{10}$. 
Among the basic classes, only $\F_5$ can contain such manifolds.
\end{remark}

Let us note that manifolds of the class $\F_5$ are counterparts of $\bt$-Kenmotsu manifolds in the case of almost contact metric manifolds.

If $\xi$ is torse-forming for $\M\in\F_5$, then it is valid the following
\begin{equation}\label{tfF5}
\begin{array}{l}
		\left(\n_x \f \right)y=-f\{g(x,\f y)\xi+\eta(y)\f x\}.
\end{array}
\end{equation}

Taking into account \remref{rem:Sl} and \remref{rem:tf}, we obtain the following
\begin{proposition}
The Reeb vector field $\xi$ of any Sasaki-like manifold $\M$ is not a torse-forming vector field.
\end{proposition}

From \eqref{roEl2} and \eqref{tf}, similarly to (ii) from \corref{cor:ElSl}, we get
\begin{corollary} Let $\M$ with torse-forming $\xi$ be Einstein-like. 
Then, it is Ricci-symmetric if and only if it is an Einstein manifold.  
\end{corollary}

\subsubsection{Example of an almost contact B-metric manifold with a torse-forming Reeb vector field}\label{ex2}

In \cite{HMMek}, it is given a 3-dimensional  $\F_5$-manifold on a Lie group $L$ with
 a basis $\{e_0, e_1, e_2\}$ of left invariant vector fields. 
The corresponding Lie algebra and the almost contact B-metric structure are determined as follows:
\begin{equation}\label{strL}
\begin{array}{c}
[e_0,e_1]=p e_1, \qquad [e_0,e_2]=p e_2, \qquad
[e_1,e_2]=0,\qquad p\in\R,\\
\begin{array}{l}
\f e_0=0,\qquad \f e_1=e_{2},\qquad \f e_{2}=- e_1,\qquad \xi=
e_0,\\
\eta(e_0)=1,\qquad \eta(e_1)=\eta(e_{2})=0,
\end{array}\\
\begin{array}{l}
  g(e_0,e_0)=g(e_1,e_1)=-g(e_{2},e_{2})=1, \\
  g(e_0,e_1)=g(e_0,e_2)=g(e_1,e_2)=0.
\end{array}
\end{array}
\end{equation}

The obtained manifold $(L,\f,\xi,\eta,g)$ is characterized as an $\F_5$-manifold by the following non-zero components
\begin{equation}\label{FijkF5}
\begin{array}{c}
\n_{e_1}e_1=-\n_{e_2}e_2=p e_0,\qquad  
\n_{e_1}e_0=-p e_1,\qquad 
\n_{e_2}e_0=-p e_2\\
F_{102}=F_{120}=F_{201}=F_{210}=\frac{1}{2}\ta^*_0=-p.
\end{array}
\end{equation}
%
Also, there are computed the components $R_{ijkl}$ of $R$   
and $\rho_{ij}$ of $\rho$.
The non-zero of them are determined by the following equalities  and the property $R_{ijkl}=-R_{jikl}=-R_{ijlk}$: 
\begin{equation*}\label{res}
\begin{array}{l}
-R_{0101}=R_{0202}=R_{1212}=\frac12\rho_{00}=\frac12\rho_{11}
=-\frac12\rho_{22}=-p^2.
\end{array}
\end{equation*}
Moreover, the values of 
$\tau$ 
	and the sectional curvatures $k_{ij}=k(e_i,e_j)$ are 
\begin{equation*}\label{curvF5}
\begin{array}{l}
	\frac16 \tau=k_{12}=k_{01}=k_{02}=-p^2.
\end{array}
\end{equation*}
It is shown that the constructed manifold is Einstein because $\rho=-2p^2 g$ holds, \ie \eqref{defEl} is valid for constants 
\begin{equation}\label{abcF5}
(a,b,c)=(-2p^2,0,0).
\end{equation}

According to the results for $\n\xi$ in \eqref{FijkF5}, we make sure that $\xi$ is a torse-forming vector field with constant 
\begin{equation}\label{fF5}
f=-p.
\end{equation}

\section{Ricci-like solitons on almost contact B-metric manifolds}\label{sect-1}
%

A pseudo-Riemannian manifold $(M,g)$ is called a \emph{Ricci soliton} if it admits a smooth non-zero vector field $v$ 
on $M$ such that \cite{Ham82}
\begin{equation*}\label{Rs}
\begin{array}{l}
\frac12 \mathcal{L}_v g + \rho + \lm\, g =0, 
\end{array}
\end{equation*}
where $\mathcal{L}$ denotes the Lie derivative, 
$\rho$ stands for the Ricci tensor field and $\lm$ is a constant.  
A Ricci soliton is called \emph{shrinking}, \emph{steady} or \emph{expanding} according to whether $\lm$ is negative, zero or positive, respectively \cite{ChoLuNi}.
In the trivial case when $v$ is a Killing vector field,  
the Ricci soliton is an Einstein manifold. 

A generalization of the Ricci soliton on an almost contact metric manifold $(M,\f,\xi,\eta,g)$ for example is 
the \emph{$\eta$-Ricci soliton}  defined 
by \cite{ChoKim}
\begin{equation*}\label{eRs}
\begin{array}{l}
\frac12 \mathcal{L}_v g  + \rho+ \lm\, g  + \nu\, \eta\otimes\eta=0, 
\end{array}
\end{equation*}
where $\nu$ is also a constant. Obviously, an $\eta$-Ricci soliton with $\nu=0$ is a Ricci soliton.

Further in the present paper, $(M,\f,\xi,\eta,g)$ stands for an almost contact B-metric manifold.

Due to the presence of two associated metrics $g$ and $\g$ on $M$, 
we introduce a generalization of the Ricci soliton and the $\eta$-Ricci soliton 
as follows. 
\begin{definition}
An almost contact B-metric manifold $\M$ is called a
\emph{Ricci-like soliton} with potential vector field $\xi$ if its Ricci tensor $\rho$ satisfies the following condition for a triplet of constants $(\lm,\mu,\nu)$
\begin{equation}\label{defRl}
\begin{array}{l}
\frac12 \mathcal{L}_{\xi} g  + \rho + \lm\, g  + \mu\, \g  + \nu\, \eta\otimes \eta =0. 
\end{array}
\end{equation}
\end{definition}

For the Lie derivative of $g$ along $\xi$ we have the following expression
\begin{equation}\label{Lg=nxi}
\left(\mathcal{L}_{\xi} g \right)(x,y)=g(\n_x \xi,y)+g(x,\n_y \xi)
\end{equation}
and then \eqref{defRl} implies for the scalar curvature the following
\[
\tau=-\Div\xi-(2n+1)\lm -\mu-\nu.
\]

\begin{proposition}\label{prop-RlEl}
Let $\M$ with a Ricci tensor $\rho$ be Einstein-like with constants $(a,b,c)$ and let the manifold admit a Ricci-like soliton with potential $\xi$ and constants $(\lm,\mu,\nu)$.
Then the following properties are valid:
\begin{enumerate}
	\item[(i)] $a+b+c+\lm+\mu+\nu=0$;

	\item[(ii)] $\xi$ is geodesic, \ie $\n_{\xi}\xi=0$;

	\item[(iii)] $\left(\n_{\xi} \f\right)\xi=0$, $\n_{\xi} \eta=0$, $\om=0$;

	\item[(iv)] $\left(\n_{\xi}\rho\right)(y,z)=b\,g\bigl(\left(\n_{\xi}\f\right)y,z\bigr)$.

\end{enumerate}
\end{proposition}
\begin{proof}
Using \eqref{defEl}, \eqref{defRl} and \eqref{Lg=nxi}, we obtain 
\begin{equation}\label{ElRl}
\begin{array}{l}
g(\n_x \xi,y)+g(\n_y \xi,x)=-2\{(a+\lm)g(x,y)+(b+\mu)g(x,\f y)\\
\phantom{g(\n_x \xi,y)+g(\n_y \xi,x)=-2\{}
+(b+c+\mu+\nu)\eta(x)\eta(y)\}.
\end{array}
\end{equation}
Substituting $y$ for $\xi$ in \eqref{ElRl}, we obtain (i) and immediately after that (ii).
Bearing in mind \eqref{F=nfi}, \eqref{F-prop} and \eqref{t}, the property (ii) is equivalent to the vanishing of  
$\left(\n_{\xi} \f\right)\xi$, $\n_{\xi} \eta$ and $\om$, i.e. (iii) is valid. 
On the other hand, \eqref{roEl2} and $\n_{\xi} \xi=0$ yields (iv).
\end{proof}

Taking into account the characterization of the basic classes $\F_i$, $i\in\{1,\A\dots,\A11\}$ of $\M$ by the components of $F$, given in \cite{HM1}, and the assertions (iii) and (iv) of the latter proposition, we deduce the following
\begin{corollary}
Let $\M$ satisfy the hypothesis of \propref{prop-RlEl}. Then, we have: 
\begin{enumerate}
	\item[(i)] The manifold does not belong to $\F_{11}$ or to its direct sum with other basic classes.
	\item[(ii)] The Ricci tensor is parallel along $\xi$ if and only if the manifold does not belong to $\F_{10}$ or to its direct sum with other basic classes.
	\item[(iii)] The Ricci tensor is parallel along $\xi$ if and only if the manifold is $\eta$-Einstein.
\end{enumerate}
\end{corollary}
%

\subsection{Ricci-like solitons on Sasaki-like almost contact B-metric manifolds}\label{sec:RlSl}

\begin{theorem}\label{thm:RlSl}
Let $\M$ be a $(2n+1)$-dimensional Sasaki-like manifold and let $a$, $b$, $c$, $\lm$, $\mu$, $\nu$ be constants that satisfy the following equalities:
\begin{equation}\label{SlElRl-const}
a+\lm=0,\qquad b+\mu-1=0,\qquad c+\nu+1=0.
\end{equation}
Then, the manifold admits 
a Ricci-like soliton with potential $\xi$ and constants $(\lm,\A\mu,\A\nu)$, where $\lm+\mu+\nu=-2n$,
if and only if 
it is Einstein-like with constants $(a,b,c)$, where $a+b+c=2n$.

In particular, we get:
\begin{enumerate}
	\item[(i)]    The manifold admits an $\eta$-Ricci soliton with potential $\xi$ and constants $(\lm,-2n-\lm)$ if and only if
the manifold is Einstein-like with constants $(-\lm,1,\lm+2n-1)$. 

	\item[(ii)]   The manifold admits a shrinking Ricci soliton with potential $\xi$ and constant $-2n$ if and only if 
the manifold is Einstein-like with constants $(2n,1,-1)$. 

	\item[(iii)]   The manifold is $\eta$-Einstein with constants $(a,2n-a)$  if and only if
it admits a Ricci-like soliton with potential $\xi$ and constants $(-a,1,a-2n-1)$. 

	\item[(iv)]   The manifold is Einstein with constant $2n$ if and only if
it admits a Ricci-like soliton with potential $\xi$ and constants $(-2n,1,-1)$. 
\end{enumerate}
\end{theorem}
\begin{proof}
Taking into account the expression of $\n_x \xi$ from \eqref{curSl} and applying it in \eqref{Lg=nxi} and \eqref{defRl},
we obtain 
that condition \eqref{defRl} 
is equivalent to 
\begin{equation}\label{SlRl-rho}
\rho=-\lm g+(1-\mu)\g-(1+\nu)\eta\otimes\eta,
\end{equation}
which coincides with  \eqref{defEl} under conditions \eqref{SlElRl-const}. 
The equality $a+b+c=2n$ is known from \eqref{tauElSl2}, whereas
$\lm+\mu+\nu=-2n$ is a consequence of \eqref{SlRl-rho} for $\rho(x,\xi)$ and the corresponding formula from \eqref{curSl}.
Therefore, the main assertion is valid.

The assertions (i), (ii), (iii) and (iv) are corollaries of the main assertion for the cases 
$\mu=0$, $\mu=\nu=0$, $b=0$ and $b=c=0$, respectively. 
\end{proof}



\subsubsection{Example of an Ricci-like soliton on the manifold}\label{ex1-R}
Let us return to the example of Sasaki-like almost contact B-metric manifold $(G,\f,\allowbreak{}\xi,\allowbreak{}\eta,\allowbreak{}g)$ from \S\,\ref{ex1}, which is also $\eta$-Einstein, according to \S\,\ref{ex1-E}.

Now, let us check whether it admits a Ricci-like soliton with potential $\xi$.
Using \eqref{strEx1}, \eqref{nijEx1} and \eqref{Lg=nxi}, 
we compute the components $\left(\LL_\xi g\right)_{ij}=\left(\LL_\xi g\right)(e_i,e_j)$ and
the non-zero of them are the following
\begin{equation}\label{Lex1}
\begin{array}{l}
\left(\LL_\xi g\right)_{13}=\left(\LL_\xi g\right)_{24}=\left(\LL_\xi g\right)_{31}=\left(\LL_\xi g\right)_{42}=2.
\end{array}
\end{equation}
After that, taking into account 
\eqref{strEx1}, \eqref{Rex1}  and \eqref{Lex1}, we obtain that \eqref{defRl} is satisfied for 
\begin{equation}\label{lmnS}
(\lm,\mu,\nu)=(0,1,-5)
\end{equation}
 and $(G,\f,\allowbreak{}\xi,\allowbreak{}\eta,\allowbreak{}g)$
admits a Ricci-like soliton with potential $\xi$.

\begin{remark}
The constructed 5-dimensional $\eta$-Einstein Sasaki-like manifold $(G,\f,\allowbreak{}\xi,\eta,g)$ admitting a Ricci-like soliton with potential $\xi$  
has relevant constants, determined by \eqref{abcS} and \eqref{lmnS}, that satisfy \eqref{SlElRl-const} and therefore this example
supports the case (iii) of \thmref{thm:RlSl} and the rest results in \S\,\ref{sec:RlSl}.
\end{remark}

\subsection{Ricci-like solitons with torse-forming potential on almost contact B-metric manifolds}\label{sec:Rltf}

\begin{theorem}\label{thm:Rltf}
Let $\xi$ on $\M$, $\dim M=2n+1$,  be torse-forming with function $f$.
The manifold admits 
a Ricci-like soliton with potential $\xi$ and constants $(\lm,\mu,\nu)$
if and only if 
the manifold is Einstein-like with constants $(a,b,c)$
provided that
$f$ is a constant and the following equalities are satisfied:
\begin{equation}\label{tfElRl-const}
a+\lm+f=0,\qquad b+\mu=0,\qquad c+\nu-f=0.
\end{equation}

In particular, we have:
\begin{enumerate}
	\item[(i)]    The manifold admits an $\eta$-Ricci soliton with potential $\xi$ and constants $(\lm,\nu)$ if and only if
the manifold is $\eta$-Einstein with constants $(a,c)$, where the equality  $a+c+\lm+\nu=0$ is valid. 

	\item[(ii)]   The manifold admits a Ricci soliton with potential $\xi$ and constant $\lm$ if and only if 
the manifold is $\eta$-Einstein with constants $(-\lm-f,f)$. 

	\item[(iii)]   The manifold is Einstein with constant $a$ if and only if
it admits an $\eta$-Ricci with potential $\xi$ and constants $(-a-f,f)$. 
\end{enumerate}
\end{theorem}

\begin{proof}
If we have a Ricci-like soliton with torse-forming potential $\xi$ on $\M$, 
then the first equality of \eqref{tf}, \eqref{defRl} and \eqref{Lg=nxi}  yield the form of its Ricci tensor as follows
\begin{equation*}\label{tfRl-rho}
\rho=-(\lm+f) g-\mu\g-(\nu-f)\eta\otimes\eta.
\end{equation*}
According to  \eqref{defEl}, the latter equality shows that the manifold is Einstein-like
if and only if $f$ is a constant. In this case 
\eqref{tfElRl-const} are valid.

Vice versa, let $\M$ have torse-forming $\xi$ with function $f$ and let the manifold be Einstein-like with constants $(a,b,c)$.
Then, applying \eqref{tf} and \eqref{defEl} in \eqref{Lg=nxi} and \eqref{defRl} for constants $(\lm,\mu,\nu)$, 
we establish that condition \eqref{defRl} is satisfied if and only if \eqref{tfElRl-const} holds. Therefore, $f$ is a constant.

The particular cases (i), (ii) and (iii) follow from the main assertion for  
$b=\mu=0$, $b=\mu=\nu=0$ and  $b=c=\mu=0$, respectively. 
\end{proof}

\begin{proposition}\label{prop-RlEltf}
Let $\M$ be Einstein-like with constants $(a,b,c)$ and let it admit a Ricci-like soliton with constants $(\lm,\mu,\nu)$ and a potential $\xi$ which is torse-forming with constant $f$.
Then $f$ is determined by:
\begin{equation}\label{tf1}
f=\varepsilon\sqrt{-\frac{a+b+c}{2n}}=\varepsilon\sqrt{\frac{\lm+\mu+\nu}{2n}}, \qquad \varepsilon=\pm 1.
\end{equation}
\end{proposition}

\begin{proof}
 %
%
%
The expression of $\n\xi$ in \eqref{tf} for contant $f$ implies  
\begin{equation}\label{Rxif}
R(x,y)\xi=f^2\{\eta(x)y-\eta(y)x\},\qquad \rho(x,\xi)=-2nf^2\eta(x).
\end{equation}
Comparing the second equality of \eqref{Rxif} 
to the fourth one in \eqref{roEl1}, we get the equation
$-2n f^2=a+b+c$. The latter equality and (i) of \propref{prop-RlEl} imply \eqref{tf1}.
\end{proof}

The first equality of \eqref{Rxif} implies the following
\begin{remark}
The $\xi$-sections $\al=\Span\{x,\xi\}$ of $\M$, where $\xi$ is torse-forming with constant $f$, 
have a constant negative sectional curvature  $k(\al)=-f^2$.
\end{remark}

The following corollary follows from the expression $\tau=(2n+1)a$ for an Einstein manifold, \propref{prop-RlEltf} and assertion (iii) of \propref{thm:Rltf}. 
\begin{corollary}\label{cor:E}
The manifold  $\M$, $\dim M=2n+1$, is Einstein with a negative scalar curvature $\tau$ if and only if 
the manifold admits an $\eta$-Ricci soliton with potential $\xi$ 
and constants $(\lm=-\frac{\tau}{2n+1}-f, \nu=f)$, 
where $f$ is the constant for the torse-forming  vector field $\xi$ 
and $f=\varepsilon\sqrt{\frac{-\tau}{2n(2n+1)}}$, $\varepsilon=\pm 1$.
\end{corollary}

For the obtained $\eta$-Ricci soliton in \corref{cor:E}, it is clear that the sign of $\nu$ depends on $\varepsilon$.  
Then, we differ the following cases:

1) If $\tau<-1-\frac{1}{2n}$, then the both values of $\lm$  are positive, whereas $\nu$ is either positive or negative. 

2) If $\tau=-1-\frac{1}{2n}$, then $(\lm,\nu)$ is either $\left(\frac{1}{n},-\frac{1}{2n}\right)$ 
or $\left(0,\frac{1}{2n}\right)$.

3) If $-1-\frac{1}{2n}<\tau<0$, then  one value of $\lm$ is positive and the other is negative, 
with the sign of $\nu$ opposite that of $\lm$.



Since any $\F_5$-manifold has a torse-forming vector field $\xi$, we conclude the following
\begin{proposition}\label{prop-RlEltfRs}
Let $\M\in\F_5\setminus\F_0$ be Einstein-like with constants $(a,b,c)$ and let it admit a Ricci-like soliton with constants $(\lm,\mu,\nu)$ and a potential $\xi$. Then we have:
\begin{enumerate}

\item[(i)]  It is Ricci-symmetric if and only if it is an Einstein manifold.  

\item[(ii)] Its Ricci tensor  is $\eta$-parallel  and parallel along $\xi$. 
\end{enumerate}
\end{proposition}

\begin{proof}

By virtue of \eqref{roEl2}, \eqref{tf}, \eqref{tfF5} and the first equality of \eqref{tfElRl-const}, we get
\begin{equation*}\label{roElF5}
\begin{array}{l}
\left(\n_x\rho\right)(y,z)=(a+\lm)\bigl\{(b+c)\left\{g(\f x, \f y)\eta(z) +g(\f x, \f z)\eta(y)\right\} \\
\phantom{\left(\n_x\rho\right)(y,z)=(a+\lm)\bigl\{(b}
+b\left\{g(x,\f y)\eta(z) +g(x,\f z)\eta(y)\right\}\bigr\}.
\end{array}
\end{equation*}
Using the latter expression, in a similar way as in (ii) and (iii) of \corref{cor:ElSl}, we establish the truthfulness of the assertions.
\end{proof}

\subsubsection{Example of a Ricci-like soliton on an almost contact B-metric manifold with torse-forming $\xi$}\label{ex2-R}

Let us return to the $\F_5$-manifold $(L,\f,\xi,\eta,g)$ from \S\,\ref{ex2}, for which $\xi$ is a torse-forming vector field
with constant 
$f=-p$.

Using \eqref{FijkF5} and \eqref{Lg=nxi}, we compute the components $\left(\LL_\xi g\right)_{ij}$  and
the non-zero of them are the following  
\begin{equation}\label{Lex2}
\begin{array}{l}
\left(\LL_\xi g\right)_{11}=-\left(\LL_\xi g\right)_{22}=-2p.
\end{array}
\end{equation}

Taking into account 
\eqref{strL}  and \eqref{Lex2}, we obtain that \eqref{defRl} is satisfied for constants
\begin{equation}\label{lmnF5}
(\lm,\mu,\nu)=(p+2p^2,0,-p) 
\end{equation}
and $(L,\f,\allowbreak{}\xi,\allowbreak{}\eta,\allowbreak{}g)$
admits a Ricci-like soliton with torse-forming potential $\xi$.

\begin{remark}
The constructed Einstein manifold $(L,\f,\xi,\eta,g)$ admitting an $\eta$-Ricci soliton with torse-forming potential $\xi$  
has relevant constants, determined by \eqref{abcF5}, \eqref{fF5} and \eqref{lmnF5}, that satisfy \eqref{tfElRl-const} and therefore this example
supports the case (iii) of \thmref{thm:Rltf} and the rest results in \S\,\ref{sec:Rltf}.
\end{remark}

\subsection*{Acknowledgment}
The author was supported by projects MU19-FMI-020 and FP19-FMI-002 of the Scientific Research Fund,
University of Plovdiv Paisii Hilendarski, Bulgaria.

\end{document}